\documentclass[11pt]{article}

\usepackage{amsmath}
\usepackage{amssymb}
\usepackage{latexsym}
\usepackage{url}

\newenvironment{proof}{\medbreak\noindent{\it Proof.}\rm}{\hfill$\square$\rm}

 \newcommand{\beq}{\begin{equation}}
\newcommand{\eeq}{\end{equation}}

\newenvironment{proof*}{\vskip 2mm\noindent {}}{\hfill $\Box$ \vskip 2mm}

\newcommand{\Cn}{{\mathbb  C\sp n}}

\newcommand{\N}{{\mathbb  N}}

\newcommand{\I}{{\mathcal I}}
\newcommand{\J}{{\mathcal{J}}}

\newcommand{\CO}{{\mathcal{O}}}

\newcommand{\PSH}{{\operatorname{PSH}}}

\renewcommand{\dim}{{\operatorname{dim}}}

\newtheorem{theorem}{Theorem}[section]
\newtheorem{cor}[theorem]{Corollary}

\newtheorem{prop}[theorem]{Proposition}

\begin{document}

\begin{center}
{\Large\bf A remark on symbolic powers}
\end{center}

\begin{center}
{\large Alexander Rashkovskii}
\end{center}

\vskip1cm

\begin{abstract} A non-local version of the Brian\c{c}on-Skoda theorem is obtained. In particular, the symbolic powers $\I^{(p)}$ of a zero dimensional radical ideal $\I\subset \CO(X)$ generated by holomorphic functions $g_1,\ldots,g_r\in\CO(Y)$  on a Stein manifold $X\Subset Y$ are shown to satisfy $\I^{(p+q)}\subset \I^p$ for $q=\min\{\dim\, X,r-1\}$ and all $p\in\N$, which contributes to the so-called containment problem.

\medskip\noindent
{\sl Mathematic Subject Classification}: 32E25, 13A10
\end{abstract}

\section{Introduction}

Recall that symbolic powers $\I^{(p)}$ of a radical ideal $\I\subset\CO(X)$ on a complex manifold $X$ are composed by the functions vanishing to order at least $p$ on each irreducible component of the variety $|\I|$ of the ideal. Evidently, the ordinary powers $\I^p$ are related to the symbolic ones by $ \I^p \subset \I^{(p)}$ for any $p\in\N$. Concerning relations in the opposite direction (the {\sl containment problem}), a significant activity  was aimed mainly at finding the smallest $k\ge 1$ such that
\beq\label{ELS} \I^{(kp)} \subset \I^{p}\eeq
(see, e.g., \cite{Sw}, \cite{ELS},  \cite{HH}, \cite{CHHT}, \cite{LS}, \cite{BH}; a detailed overview and further references are given in \cite{SS}). In particular, it was proved in \cite{ELS} that one can take $k$ equal to the largest codimension of the components of $|\I|$.

On the other hand, a corresponding question on relating local ideals $\J$ of holomorphic germs to their integral closures $\overline\J$ was resolved (long before the aforementioned results) by the Brian\c{c}on-Skoda theorem (\cite{BS}, see also Thm.~10.4 at Chapter VIII of \cite{Dbook}) in terms of an additive constant. Namely, for an ideal $\J\subset\CO_0(\Cn)$ generated by germs $g_1,\ldots,g_r$, one has
\beq\label{BS}\overline\J^{(p+q)}\subset\J^p,\quad  p= 1,2,\ldots,\eeq
where $q=\min\{n,r-1\}$ and
$\overline\J^{(k)}$ is the ideal of germs $f$ satisfying $|f|\le C\,|g|^k$, and $g=(g_1,\ldots,g_r)$. The main ingredient of an analytic proof of that result is the following Skoda's theorem on ideal generation.

\medskip
{\bf Theorem A} (\cite[Thm. 1]{S}, see also  Thm.~9.10 at Chapter VIII of \cite{Dbook}) {\sl Let $\Omega$ be a pseudoconvex domain in $\Cn$, $\psi\in\PSH(\Omega)$, and let $f, g_1,\ldots,g_r\in\CO(\Omega)$ satisfy
$$ \int_\Omega |f|^2 |g|^{-2(\alpha q+1)}e^{-\psi}\,dV=I<\infty$$
for some $\alpha>1$ and $q=\min\{n,r-1\}$. Then there exists $h_1,\ldots,h_r\in\CO(\Omega)$ such that $f=\sum g_jh_j$ and
$$ \int_\Omega |h|^2 |g|^{-2\alpha q}e^{-\psi}\,dV<\frac{\alpha}{\alpha-1}\,I.$$
}

\medskip

In this short note, we show that (\ref{BS}) readily extends, by the same method, to ideals $\I\subset\CO(\Omega)$ whose generators are  holomorphic in a neighbourhood of $\overline\Omega$, giving thus a non-local version of the Brian\c{c}on-Skoda theorem.
 This observation might be known, however we were unable to find explicit references.
  When the ideal $\I$ is, in addition, radical and zero dimensional, the condition $|f|\le C\,|g|^k$ means precisely that the multiplicity of $f$ at each point of the set $g^{-1}(0)$ is at least $k$, so the ideals $\overline\I^{(k)}$ are just the symbolic powers $\I^{(k)}$. This is not true for radical ideals of positive dimension because in the definition of the symbolic powers no condition is put on $f$ on the singular set of the variety, so, in general, $\overline\I^{(k)}\subset \I^{(k)}$.

The case of zero dimensional varieties $|\I|$ is considered in Section~\ref{sec2}. We were actually primarily interested in this case because of considerations in \cite{RT} where the powers of ideals of holomorphic functions vanishing on finite sets were studied, see a remark after Corollary~\ref{corgap}. Note that for zero dimensional radical ideals, the inclusions $\J^{(p+q)}\subset\J^p$ imply (\ref{ELS})
for any $k\ge 1+\frac{n}{2}$ and $p\in\N$, which improves for this case the aforementioned result from  \cite{ELS}.

The case $\dim\,|\I|>0$ is treated in Section~\ref{sec3}. Note that without assuming the generators being holomorphic in a neighbourhood of $\Omega$, the inclusions $\overline\I^{(p+n)} \subset\I^p$ for any $p\in\N$, in the sense of sheaves, follow directly from the Brian\c{c}on-Skoda theorem.

The author thanks Tomasz Szemberg for attracting our attention to the survey \cite{SS}. The author is grateful to Lawrence Ein, Tai Ha, and Robert Lazarsfeld for pointing out that in a previous version, corollaries of the main result were erroneously extended to radical ideals of positive dimension.

\section{Zero dimensional case}\label{sec2}

In this section, $S$ is a finite subset of a bounded pseudoconvex domain $\Omega\subset\Cn$ and $\I$ is an ideal of functions from $\CO(\Omega)$ whose variety $|\I|$ (common set of zeros) coincides with $S$. By Cartan's theorem, $\I$ has finitely many generators $g_1,\ldots,g_r\in\CO(\overline\Omega)$ and, moreover, $g^{-1}(0)\cap\overline\Omega=S$, where $g=(g_1,\ldots,g_r)$. Let $\overline\I^{(p)}$, $p>0$, be the collections of functions $f\in\CO(\Omega)$ such that $\log|f(z)|\le p\,\log|g(z)| +O(1)$ as $z\to S$.

Denote also
$$ \widehat\I^{(p)}=\{f\in\CO(\Omega):\:  |f||g|^{-p}\in L^2(\Omega')\ \forall
\Omega'\Subset\Omega\}, \quad p>0.
$$
Evidently, $\overline\I^{(p)} \subset \widehat\I^{(p)}$.
If $f\in \widehat\I^{(p)}$, then for each point $a\in S$ we have $|f|\,|g|^{-(p+\epsilon)}\in L^2_{loc}(a)$ with some $\epsilon>0$ (which follows, for example, from Demailly's strong openness conjecture proved in \cite{GZ}), and of course this is true for any $\epsilon>0$ if $a\not\in S$. Then, by the compactness of $S$, we have
\begin{equation}\label{eps} \int_{\Omega'} |f|^2|g|^{-2(p+\epsilon)}\,dV<\infty \end{equation}
for some $\epsilon>0$ and all $\Omega'\Subset\Omega$.

\begin{theorem}\label{prop1} If $|\I|\Subset\Omega$, then $\widehat\I^{(p+1)}=\I\cdot \widehat\I^{(p)}$ for any $p\ge q=\min\{n,r-1\}$ and, by induction, $\overline\I^{(p+q)} \subset \widehat\I^{(p+q)}\subset\I^p$ for any $p\in\N$.
\end{theorem}

\begin{proof} Evidently, $\I\cdot \widehat\I^{(p)}\subset \widehat\I^{(p+1)}$, so it suffices to show the reverse inclusion.
Take any $f\in \widehat\I^{(p+1)}$. Then there exists a non-negative function $\psi\in \PSH(\Omega)$ (for example, $\psi=|f|$) such that
$$ \int_{\Omega} |f|^2|g|^{-2(p+1+\epsilon)}e^{-\psi}\,dV=I<\infty$$
with $\epsilon>0$ as in (\ref{eps}) with $p$ replaced by $p+1$. By Theorem~A, there exist $h_1,\ldots,h_r\in\CO(\Omega)$ such that
$f=\sum_j g_jh_j$ and
$$ \int_{\Omega} |h|^2|g|^{-2(p+\epsilon)}e^{-\psi}\,dV<\frac{\alpha}{\alpha-1}\,I $$
with $\alpha=(p+\epsilon)/q$. Therefore, for any $\Omega'\Subset\Omega$, we have
$$ \int_{\Omega'} |h|^2|g|^{-2(p+\epsilon)}\,dV\le  e^{\sup_{\Omega'}\psi} \int_{\Omega} |h|^2|g|^{-2(p+\epsilon)}e^{-\psi}\,dV< \infty,
$$
 so $h_1,\ldots,h_r\in \widehat\I^{(p)}$ and $f\in \I\cdot \widehat\I^{(p)}$.
\end{proof}

\bigskip
By induction, Theorem~\ref{prop1} implies

\begin{theorem}\label{BSglobal} If $|\I|\Subset\Omega$, then $\overline\I^{(p+q)} \subset \widehat\I^{(p+q)}\subset\I^p$ for any $p\in\N$.
\end{theorem}

For radical ideals, this gives us

\begin{cor}\label{corrad} Let $\I$ be a radical ideal in $\CO(\Omega)$ whose variety is a finite set. Then its symbolic powers $\I^{(p)}$ satisfy $\I^{(p+q)} \subset \I^p$ for any $p\in\N$, where $q=\min\{n,r-1\}$ and $r$ is the number of generators of $\I$.
\end{cor}

This allows us to compute the asymptotical behaviour of the lengths of the (usual) powers of the ideals of functions vanishing on finite sets. Recall that the \emph{length} of such an ideal $\mathcal J$ is $l (\mathcal J)= \dim\, \mathcal O/\mathcal J<\infty$.

\begin{cor}\label{corgap} Let $\I\subset\CO(\Omega)$ be the ideal of functions vanishing at $N$ given points of $\Omega$. Then the lengths $l(\I^p)$ of the powers $\I^p$ of the ideal $\I$ satisfy
\beq\label{poles}
\lim_{p\to\infty} n!\,p^{-n} l(\I^p)=N.
\eeq
\end{cor}

\begin{proof}
Each ideal $\I^{(p)}$ consists of all functions $f\in{\CO}(\Omega)$ satisfying
$$\frac{\partial^{|\beta|} f}{\partial z^{\beta}}(a_j)=0,\quad |\beta|<p,\  1\le j\le N,$$
so
$$l(\I^{(p)})=\binom{p+n-1}{n}N.$$
Since $\I^{(p+q)} \subset\I^p\subset \I^{(p)}$, we have
$$ \binom{p+n-1}{n}\,N\le l(\I^p)\le \binom{p+q+n-1}{n}\,N,$$
which implies the claim.
\end{proof}

\medskip{\it Remark.} In \cite{RT}, a problem of behavior of multipole Green functions was studied for the case when the poles stretch to a single point. In \cite[Prop. 4.7]{RT}, it was claimed that the lengths of the ideals $\I^p$ has the asymptotic property (\ref{poles}), while actually this was proved there for the lengths of the symbolic powers $\I^{(p)}$ instead. Thus Corollary~\ref{corgap} fills the gap.

\section{Positive dimensional varieties}\label{sec3}
If $|\I|$ is not compact (i.e., $\dim\, |\I|>0$ or $\dim\,|\I|=0$ with $\sharp |\I|=\infty$), then for $f\in \widehat\I^{(p)}$ we can no longer guarantee the existence of $\epsilon>0$ such that (\ref{eps}) holds for all $\Omega'\Subset\Omega$.
We can however do that for each subdomain $\Omega'$ individually, with $\epsilon=\epsilon(f,\Omega')$. Therefore, in this situation it is natural  to work essentially locally, that is, with the coherent ideal sheaves $\I=\{\I_x\}_{x\in X}\subset\CO_X$, defined in terms of {\it local generators} $g_1,\ldots,g_r\in\CO_x$ (with $r$ eventually depending on $x$), where $X$ is a Stein manifold of dimension $n$. For an open subset $U$ of $X$, let $\I_U$ denote the space of all holomorphic functions in $U$ with germs in $\I$. We keep the notation
$\I^{(p)}$ , $\overline\I^{(p)}$, and $\widehat\I^{(p)}$ for the corresponding sheaves defined as in the previous section (with $S$ replaced by $|\I|$).

In this setting, directly from  Theorem~A with $\psi=0$, we get that every $f\in \widehat\I_{U}^{(p)}$, where $U\Subset X$,  can be represented as $f=\sum_j g_jh_j$ with functions $h_1,\ldots, f_r\in \CO(U) $ satisfying
$$ \int_{U} |h|^2|g|^{-2(p+\epsilon)}\,dV <\infty$$
for some $\epsilon>0$ (depending on $f$ and $U$).
Therefore, $\widehat\I_{U}^{(p+1)}\subset \I_{U}\cdot \widehat\I_{U}^{(p)}$ for all $p\ge n$ and any $U\Subset X$, which means $\widehat\I^{(p+1)}\subset \I\cdot \widehat\I^{(p)}$ in the sense of sheaves. Since $\I\cdot \widehat\I^{(p)}\subset \widehat\I^{(p+1)}$ and $\overline\I^{(p)} \subset \widehat\I^{(p)}$, this implies

\begin{prop}\label{theosh} Let  $\overline\I^{(p)}$ and $\widehat\I^{(p)}$ be ideal sheaves defined as above for a coherent ideal sheaf $\I\subset\CO_X$ and $p>0$. Then, for any $p\ge n=\dim\,X$, we have $\widehat\I^{(p+1)}=\I\cdot \widehat\I^{(p)}$ and, consequently,
$\overline\I^{(p+n)} \subset \widehat\I^{(p+n)}\subset\I^p$ for any $p\in\N$.
\end{prop}

A global result can be deduced by assuming $X$ to be compactly supported in a complex manifold $Y$ and $\I$ to have global generators $g_1,\ldots,g_r\in \CO(Y)$; denote the ideal generated by $(g_j)$ in $\CO(Y)$ by $\J$. Let $\mu$ be a log resolution of $\J$ over $Y$, i.e., a proper surjective holomorphic map from a smooth manifold $\widetilde Y$ to $Y$ which is an isomorphism outside $\mu^{-1}(|\J|)$ and such that $\widetilde\J=\mu^*\J=O_{\widetilde Y}(-D)$, where $D$ is a normal crossing divisor. Furthermore, given $f\in \widehat\I^{(p)}$, the log resolution $\mu$ can be chosen such that $F=D_{\widetilde X}+E_{\widetilde X}+\mu^*Z_f$ is a normal crossing divisor in $\widetilde X=\mu^{-1}(X)$; here $E$ is the exceptional divisor of $\mu$.

Let
$$K_{\widetilde X}=\mu^*K_X+\sum_j a_jF_j, \quad D_{\tilde X}=\sum_j b_j F_j, \quad \mu^*Z_f =\sum_j c_j F_j$$
$F_j$ being irreducible components of $|F|$. Note that $b_j\neq 0$ only for finitely many $j$; denote the collection of such indices $j$ by $J$. Then, by the coordinate change formula, the condition $f\in \widehat\I^{(p)}$ for  $p>0$ is equivalent to the collection of bounds
\beq\label{bnds} a_j+ c_j -p\,b_j>-1, \quad j=1,2,\ldots.\eeq
Since such an inequality is always true if $j\notin J$, system (\ref{bnds}) reduces to finitely many inequalities
$$ p < \frac{a_j+c_j+1}{b_j},\quad j\in J,$$
and so, implies
$$ \int_{X'} |f|^2|g|^{-2(p+\epsilon)}\,dV<\infty $$
for some $\epsilon>0$ and all $X'\Subset X$.

By repeating then the arguments of the proof of Theorem~\ref{prop1}, we get

\begin{theorem}\label{theoglob} Let an ideal $\I\subset\CO(X)$ on an $n$-dimensional Stein manifold $X\Subset Y$ have generators $g_1,\ldots,g_r\in\CO(Y)$. Then for any $p\ge q=\min\{n,r-1\}$, we have $\widehat\I^{(p+1)}=\I\cdot \widehat\I^{(p)}$ and, consequently,
$\overline\I^{(p+q)} \subset \widehat\I^{(p+q)}\subset\I^p$ for any $p\in\N$.
\end{theorem}

\vskip1cm

Tek/Nat, University of Stavanger, 4036 Stavanger, Norway

\vskip0.1cm

{\sc E-mail}: alexander.rashkovskii@uis.no

\end{document}